\documentclass[12]{article}
\usepackage{authblk}

\usepackage{amssymb, amsmath, amsthm}
\usepackage{ stmaryrd } 
\usepackage{color, enumerate}
\usepackage{epsfig, graphicx, subcaption, float}
\usepackage{algorithm, algorithmic}

\usepackage{comment}

\usepackage{hyperref} 

\theoremstyle{plain} 
  \newtheorem{theorem}{Theorem}[section]
  \newtheorem{proposition}[theorem]{Proposition}
  \newtheorem{lemma}[theorem]{Lemma}

\theoremstyle{definition} 
  \newtheorem{definition}[theorem]{Definition}
\theoremstyle{remark} 
  \newtheorem{remark}[theorem]{Remark}
  
\providecommand{\keywords}[1]
{
  \small\noindent	
  \textbf{{Keywords:}} #1
}

\newcommand{\bR}{\mathbb{R}}

\newcommand{\bN}{\mathbb{N}}

\newcommand{\mcP}{\mathcal{P}}

\newcommand{\cA}{\mathcal{A}}
\newcommand{\cE}{\mathcal{E}}
\newcommand{\cF}{\mathcal{F}}
\newcommand{\cG}{\mathcal{G}}

\newcommand{\cN}{\mathcal{N}}
\newcommand{\cP}{\mathcal{P}}
\newcommand{\cV}{\mathcal{V}}

\definecolor{blue}{rgb}{0,0,0.9}

\definecolor{red}{rgb}{0.9,0,0}

\title{Network Consensus in the Wasserstein Space of Probability Measures Defined on Multi-Dimensional Euclidean Spaces}

\author[1]{Pilgyu Jung}
\affil[1]{Department of Mathematics\\ \newline
Sungkyunkwan University\\ ({\tt pilgyujung@skku.edu})}
\author[2]{Yoon Mo Jung}
\affil[2]{Department of Mathematics\\ \newline
Sungkyunkwan University\\ ({\tt ymjung@skku.edu})}

\begin{document}

\maketitle

\begin{abstract}
The consensus problem—achieving agreement among a network of agents—is a central theme in both theory and applications. Recently, this problem has been extended from Euclidean spaces to the space of probability measures, where the natural notion of averaging is given by the Wasserstein barycenter. While prior work established convergence in one dimension, the case of higher dimensions poses additional challenges due to the curved geometry of Wasserstein space. In this paper, we develop a framework for analyzing such consensus algorithms by employing a Wasserstein version of Jensen’s inequality. This tool provides convexity-type estimates that allow us to prove convergence of nonlinear consensus dynamics in the Wasserstein space of probability measures on $\mathbb{R}^d$.
\end{abstract}


\section{Introduction}
The \emph{consensus problem}—achieving agreement among a network of agents—plays a crucial role in various applications and serves as a fundamental concept in both theoretical and practical contexts; see, for instance, \cite{berger1981necessary,degroot1974reaching, tsitsiklis1984problems,winkler1968consensus}.  
In particular, \cite{jadbabaie2003coordination,vicsek1995novel} demonstrated that a group of agents following nearest-neighbor interaction rules will eventually align to a common heading, provided the interaction topology becomes connected sufficiently often.
Furthermore, the works \cite{moreau2004stability, olfati2007consensus, olfati2004consensus} established consensus convergence under various conditions, including time-varying communication graphs, communication delays, and switching topologies.

In modern applications, agent states may evolve on a \emph{non-Euclidean space}.  
For example, synchronization problems (closely related to consensus \cite{li2009consensus}) typically aim to steer a network of oscillators toward a common frequency or phase.  
Such problems often involve dynamics on the circle; see \cite{acebron2005kuramoto, dorfler2014synchronization,rodrigues2016kuramoto,strogatz2000kuramoto} for details. Several noteworthy works of consensus protocols in non-Euclidean spaces are presented in \cite{MR3101606,MR3373939,MR3340784,MR2480126,tuna2007consensus}.

 

The Wasserstein barycenter provides a geometry-aware notion of averaging probability measures and has proven useful in diverse applied domains. In the field of computer graphics, it has been successfully utilized to naturally blend multiple textures \cite{rabin2011wasserstein} and to smoothly interpolate between 3D shapes \cite{solomon2015convolutional}. In machine learning, it offers an effective solution for domain adaptation by reducing the distributional discrepancy between training and testing data \cite{courty2016optimal}, and it has played a decisive role in stabilizing the training process of generative models \cite{arjovsky2017wasserstein}. In model aggregation, it has been applied to ensemble predictive distributions \cite{dognin2019wasserstein}, to aggregate subset posteriors in scalable Bayesian inference \cite{srivastava2018wasp}, and to fuse models under transport-based geometry \cite{akash2023fusion}. 

Building on these advances, the consensus problem has recently been extended to the 
\emph{space of probability measures} \cite{bishop2014distributed,bishop2021network,cisneros2022distributed}. 
In the works \cite{bishop2014distributed,bishop2021network}, a nonlinear consensus algorithm was proposed in the Wasserstein space of probability measures on $\mathbb{R}^d$, 
and it was shown that when $d=1$, the agents’ distributions converge to a common probability measure under a suitable connectivity assumption.  
Despite this progress, analyzing consensus dynamics in Wasserstein space presents significant challenges. 
A central obstacle is that the Wasserstein space $\mathcal{P}_2(\mathbb{R}^d)$ is \emph{not flat} for $d>1$; it is a positively curved metric space in the sense of Aleksandrov \cite{MR2401600}. 
This curvature implies that barycenters need not vary convexly along geodesics. Indeed, the map $\mu\to W_2^2(\mu,\nu)$ along geodesics is semi-concave. See, for instance, \cite[Chapter 7]{MR2401600}, which  
thereby complicates the convergence analysis.  
To address these issues, we employ a \emph{Wasserstein Jensen’s inequality}, 
a nonlinear analogue of Jensen’s inequality in Wasserstein space \cite{MR3590527}. 
This inequality provides convexity-type estimates in Wasserstein space and serves as a key tool for analyzing consensus algorithms.

The remainder of the paper is organized as follows. 
In Section~\ref{sec:pre}, we introduce the update rule based on the Wasserstein barycenter and state the main results. 
In Section~\ref{sec:aux}, we present several auxiliary results. 
Section~\ref{sec:proof} provides the proofs of the main results.
Finally, we provide some concluding remarks in Section~\ref{sec:conclusion}.

\section{The Consensus Protocol and Main Results}\label{sec:pre}
We begin by introducing the Wasserstein spaces and the notion of the Wasserstein barycenter.
Let $\cP(\bR^d)$ be the set of all probability measures on $(\bR^d,|\cdot |)$, where $d$ is a positive integer indicating the dimension and $|\cdot |$ is the Euclidean distance.
For $p\in [1,\infty)$,
define $\cP_p(\bR^d)=\{\mu\in\cP(\bR^d):\int_{\bR^d}|x|^p\,d\mu(x)<\infty\}$.
Let $\cP_{p,ac}(\bR^d)$  denote the set of all absolutely continuous measures in $\cP_p(\bR^d)$. 

For any two probability measures $\mu,\nu\in \cP_p(\bR^d)$, the $p$-Wasserstein distance is defined by
\[
W_p(\mu,\nu):=\min_{\gamma\in\Pi(\mu,\nu)}\left(\int_{\bR^{d}\times \bR^d}|x-y|^p\,d\gamma(x,y)\right)^{1/p},
\]
where $\Pi(\mu,\nu)$ denotes the set of transport plans from $\mu$ to $\nu$:
\[
\Pi(\mu,\nu)=\left\{\gamma\in \mcP(\bR^d\times \bR^d):\pi^1_{\#}\gamma=\mu,\pi^2_{\#}\gamma=\nu\right\}.
\]Here, $\pi^1(x,y)=x$ (resp.\ $\pi^2(x,y)=y$), and 
$\pi^1_{\#}\gamma$ (resp.\ $\pi^2_{\#}\gamma$) denotes the pushforward measure.

We define the Wasserstein barycenter as follows.
Let $E\subset \bN$ be a nonempty finite set.
Given probability measures $\{\mu_j\}_{j\in E}\in \cP_2(\bR)$ and weights $\lambda_j>0$ such that $\sum_{j\in E}\lambda_j=1$, we define
\[
\operatorname{bar}(\{(\mu_j,\lambda_j)\}_{j\in E})
:=\underset{\nu\in \cP_2(\bR^d)}{\operatorname{argmin}}\sum_{j\in E}\lambda_j W_2^2(\mu_j,\nu).
\]
If the weights are uniform, we abbreviate the barycenter as \( \operatorname{bar}(\{\mu_j\}_{j \in E}) \).

Let $\bN$ be the set of all positive integers and set $\bN_0=\bN\cup \{0\}$. We define 
\[
\llbracket k \rrbracket=\{1,2,\ldots,k\}\quad \text{for any $k \in \bN $.}
\]
For $i,j \in \bN_0$ with $i\le j$, we define 
\[
\llbracket i,j \rrbracket=[i,j]\cap \bN_0=\{i,i+1,\ldots,j\}.
\]
We now introduce an update rule based on the Wasserstein barycenter.
Consider a group of agents indexed by \(\mathcal{V} = \llbracket n\rrbracket\), and a set of (possibly time-varying) undirected links \(\mathcal{E}(t) \subset \mathcal{V} \times \mathcal{V}\), which form a network graph \(\mathcal{G}(t) = (\mathcal{V}, \mathcal{E}(t))\), where time $t$ is indexed by $\bN_0=\{0,1,2,3,\ldots\}$. Here, $n$ is a positive integer indicating the number of agents. The set of neighbors of agent \(i\) at time \(t\) is denoted as \(\mathcal{N}_i(t) = \{j \in \mathcal{V} : (i, j) \in \mathcal{E}(t)\}\). 

The adjacency matrix of the graph \(\mathbf{A}(t) \in \mathbb{R}^{n \times n}\) is symmetric and given by \(\mathbf{A}(t) = \mathbf{A}(t)^\top = [a_{ij}(t)]\), where \(a_{ij}(t) = 1\) if and only if \((i,j) \in \mathcal{E}(t)\), and \(a_{ij}(t) = 0\) otherwise. We assume that \(a_{ii}(t) = 1\) for all \(i\) and \(t\), implying that \(i \in \mathcal{N}_i(t)\) for all \(t\). 
A weighted adjacency matrix \(\mathbf{W}(t) = [w_{ij}(t)] \in \mathbb{R}^{n \times n}\) is defined such that \(a_{ij}(t) = 1\) if and only if \(w_{ij}(t) > 0\), with \(w_{ij}(t) = 0\) otherwise. Additionally, we require \(\sum_{j \in \mathcal{N}_i(t)} w_{ij}(t) = 1\).
 
Suppose that the measure of agent $i$ at time $t$ is updated by
\begin{equation}
\label{update}    \mu_i(t+1)=\operatorname{bar}(\{(\mu_j(t),w_{ij}(t))\}_{j\in \cN_i(t)}).
\end{equation}
Naturally, for the update rule \eqref{update} to converge to a consensus value, it is necessary that every pair of nodes communicates at least once over time. The following definition formalizes this requirement.

\begin{definition}
  For $j=1,2,\ldots,m$, let $\cG_j=(\cV,\cE_j)$ be a network graph with vertex set $\cV$ and link set $\cE_j$. We say that $\{\cG_1,\cG_2,\ldots, \cG_m\}$ is {\em jointly connected} if the union $\cup_{j=1}^m \cG_j=(\cV, \cup_{j=1}^m \cE_j)$ is connected. 
\end{definition}

Throughout the paper, we assume that there exists a partitioning $0=\tau_0<\tau_1<\cdots<\infty$ with $\sup_{i}|\tau_i-\tau_{i-1}|=L<\infty$ such that the collection ${\cG(\tau_i),\ldots,\cG(\tau_{i+1}-1)}$ is jointly connected for each $i$. We also assume that for $i,j \in \cV$ and $t\in \bN_0$, we have 
\begin{equation}
    \label{eq_lower}
w_{ij}(t)\ge \delta \quad \text{if $j\in \cN_i(t)$}
\end{equation}
for some constant $\delta\in (0,1)$.

\begin{remark} We note that the lower bound assumption on $w_{ij}$ in \eqref{eq_lower} encompasses 
the leaderless coordination model in \cite{jadbabaie2003coordination}. 
In that setting, the weights are given by
\[
    w_{ij}(t) =
    \begin{cases}
        \dfrac{1}{|\mathcal{N}_i(t)|}, & j \in \mathcal{N}_i(t), \\[1.2ex]
        0, & j \notin \mathcal{N}_i(t),
    \end{cases}
    \qquad \forall\, i \in \cV.
\]
\end{remark}

We are now ready to state our main results.

\begin{theorem}
    \label{thm_convergence}
    Let $\mu_1(0),\ldots,\mu_n(0)\in \cP_{2,ac}(\bR^d)$. Suppose that $\mu_i(t)$ follows the update rule \eqref{update}. Then, for all $i,j\in \cV$, we have \[\lim_{t\to \infty}W_2(\mu_i(t),\mu_j(t))=0.\]
    Moreover, there exists $\mu^*\in \cP_{2}(\bR^d)$ such that 
    \begin{equation}
        \label{main converge}
        \lim_{t\to \infty}W_p(\mu_i(t),\mu^*)=0
    \end{equation}
        for all $i \in \cV$ and all $p\in [1,2)$. 
\end{theorem}

If, in addition, the initial measures satisfy stronger integrability conditions (for example, being Gaussian or having compact support), then convergence to $\mu^*$ in $W_2$ is guaranteed.

\begin{theorem}
    \label{thm_convergence2}Let $\mu_1(0),\ldots,\mu_n(0)\in \cP_{q,ac}(\bR^d)$ for some $q\in (2,\infty)$. Then, there exists $\mu^*\in \cP_{2}(\bR^d)$ such that
    \begin{equation}
        \label{main converge2}
        \lim_{t\to \infty}W_2(\mu_i(t),\mu^*)=0,
    \end{equation}
    for all $i\in \cV$.
\end{theorem}

\section{Auxiliary results}\label{sec:aux}
In this section, we present several lemmas that will be used in the proof of the main theorem. We begin by introducing some notions of convergence.

\begin{definition}
    We say that a sequence of measures $\{\mu_k\}$ converges to a measure $\mu$ {\em narrowly} if 
    \[
    \int_{\bR^d}\varphi \,d\mu_k\to \int_{\bR^d}\varphi\, d\mu \quad \text{for any $\varphi\in C_b(\bR^d)$},
    \]
    where $C_b(\bR^d)$ is the set of all bounded continuous functions on $\bR^d$.
    We denote this convergence by $\mu_k\rightharpoonup \mu$.
\end{definition}

Next, we introduce the concept of tightness for a family of probability measures.

\begin{definition}
    Let $\mathcal{A} \subset \cP(\bR^d)$ be a family of probability measures. We say that $\mathcal{A}$ is {\em tight} if, for any $\varepsilon>0$, there exists a compact set $K_\varepsilon\subset \bR^d$ such that $\mu(\bR^d\setminus K_\varepsilon)\le \varepsilon$ for any $\mu\in \mathcal{A}$.
\end{definition}

Tightness of a family of probability measures is characterized by the following compactness criterion, known as Prokhorov’s theorem.

\begin{theorem}[Prokhorov]
\label{thm_Prokhorov}
A family $\mathcal{A}\subset \cP(\bR^d)$ is tight if and only if $\mathcal{A}$ is relatively compact with respect to narrow convergence, i.e., for any sequence $\{\mu_k\}$ in $\mathcal{A}$ there exists a subsequence $\{\mu_{k_j}\}$ and a probability measure $\mu \in \cP(\bR^d)$ such that 
\[
\mu_{k_j}\rightharpoonup \mu.
\]
\end{theorem}
\begin{proof}
    See, for instance, \cite[Theorem 2.1.11]{MR4655923}.
\end{proof}

We now introduce a condition that strengthens tightness by controlling the moments uniformly across the family.

\begin{definition}
    We say that $\cA\subset \cP(\bR^d)$ has {\em uniformly integrable $p$-moments} if 
    \[
    \lim_{R\to \infty} \int_{\bR^d\setminus B_R(0)}|x|^p\,d\mu=0 \quad \text{uniformly w.r.t. $\mu \in \cA$}.
    \]
    Here, $B_R(0)=\{x\in \bR^d:|x|<R\}$.
\end{definition}

The following lemma establishes the relationship between convergence in Wasserstein space and narrow convergence.

\begin{lemma}
\label{lem1211_1}
    Let $\{\mu_k\}\subset \cP_p(\bR^d)$. Then 
    \begin{equation*}
        \lim_{k\to \infty}W_p(\mu_k,\mu)=0\Longleftrightarrow\left\{
\begin{aligned}
    &\text{$\{\mu_k\}$ converges to $\mu$ narrowly}, and \\
    &\text{$\{\mu_k\}$ has uniformly integrable $p$-moments.}
\end{aligned}
\right.         
    \end{equation*}
    \begin{proof}
        See, for instance, \cite[Proposition 7.1.5]{MR2401600}.
    \end{proof}
\end{lemma}

The following lemma suggests a criterion for the convergence in the Wasserstein space.

\begin{lemma}
    \label{lem1211_2}
    If 
    $
    1\le p<p_1 $ and 
    \[\sup_{\mu\in\cA}\int_{\bR^d}|x|^{p_1}\,d\mu<\infty,\] then for any sequence $\{\mu_k\} \subset \cA,$ there exists a subsequence $\{\mu_{k_j}\}$ such that 
    \[
    \lim_{j\to \infty}W_p(\mu_{k_j},\mu)=0 \quad \text{for some $\mu\in \cP_{p_1}(\bR^d)$}.
    \]
\end{lemma}

\begin{proof}
    Let 
    \[    M=\sup_{\mu\in\cA}\int_{\bR^d}|x|^{p_1}\,d\mu.
    \]
    Using H\"older's inequality, we see that 
    for any $\mu\in \cA$, 
    \[
    \int_{\bR^d\setminus B_R(0)}|x|^p\,d\mu
    \le \mu(\bR^d\setminus B_R(0))^{(p_1-p)/p_1}
    \left(\int_{\bR^d\setminus B_R(0)}|x|^{p_1}\,d\mu\right)^{p/p_1}
    \]
    \[
    \le \mu(\bR^d\setminus B_R(0))^{(p_1-p)/p_1}M^{p/p_1}.
    \]
    Note that 
    \begin{equation}
        \label{eq1211_1}
    \mu(\bR^d\setminus B_R(0))\le R^{-p_1}\int_{\bR^d\setminus B_R(0)} |x|^{p_1}\,d\mu\le R^{-p_1}M.    
    \end{equation}   
    Hence, we obtain
    \begin{equation}
        \label{eq1211_2}
    \int_{\bR^d\setminus B_R(0)}|x|^p\,d\mu
    \le
    R^{p-p_1}M.
    \end{equation}
    
    Let $\{\mu_k\}$ be a sequence in $\cA$.
    Using the estimate \eqref{eq1211_1}, we see that $\cA$ is tight. By Theorem \ref{thm_Prokhorov}, there exists a narrowly convergent subsequence $\{\mu_{k_j}\}$ converging to $\mu^*$. Furthermore, by employing the inequality \eqref{eq1211_2}, it follows that $\{\mu_{k_j}\}$ has uniformly integrable $p$-moments. From Lemma \ref{lem1211_1}, we achieve that 
    \[
    \lim_{j\to \infty}W_p(\mu_{k_j},\mu^*)=0.
    \]

    To finish the proof, we check $\mu^*\in \cP_{p_1}(\bR^d)$. Note that for all $N>0$, since $\mu_{k_j}$ converges to $\mu^*$ narrowly, we have 
    \[
    \lim_{j\to\infty}\int_{\bR^d}(|x|^{p_1}\wedge N)\,d\mu_{k_j}=\int_{\bR^d}(|x|^{p_1}\wedge N)\,d\mu^*.
    \]
    By using the monotone convergence theorem and the above observation, we obtain that 
    \begin{align*}        
    \int_{\bR^d} |x|^{p_1}\,d\mu^*
    &=\limsup_{N\to \infty}\int_{\bR^d}\left(|x|^{p_1}\wedge N\right)\,d\mu^*
    \\&=\limsup_{N\to \infty}\left(\lim_{j\to \infty}\int_{\bR^d}\left(|x|^{p_1}\wedge N\right)\,d\mu_{k_j}\right)
    \\&\le \limsup_{N\to \infty} M=M<\infty.
    \end{align*}
    Hence, the lemma is proved.
    \end{proof}

Here, we introduce a notion of convexity in the Wasserstein space and a key estimate.

\begin{definition}
    For $\mu_0,\mu_1\in \cP_{2,ac}(\bR^d)$ and $t\in [0,1]$, we call the measure $\mu_t=[(1-t)\operatorname{Id}+tT]_{\#}\mu_0$, $t\in [0,1]$ {\em the displacement interpolant} between $\mu_0$ and $\mu_1$. Here, $T$ is the unique optimal transport map satisfying $T_{\#}\mu_0=\mu_1$.
\end{definition}

\begin{definition}
    A functional $\cF:\operatorname{Dom}(\cF)\subset \cP_{2,ac}(\bR^d)\to \bR\cup \{+\infty\}$ is said to be {\em $k$-displacement convex} for $k \in \bR$, if for each displacement interpolation $\mu_t$ between endpoint measures $\mu_0$ and $\mu_1$, we have 
    \[
    \cF(\mu_t)\le (1-t)\cF(\mu_0)+t\cF(\mu_1)-\frac{k}{2}t(1-t)W_2^2(\mu_0,\mu_1).
    \]
\end{definition}

We recall that $V:\bR^d \to \bR$ is a $k$-convex function for $k\in \bR$ if 
\[
V((1-t)x+ty)\le (1-t)V(x)+tV(y)-\frac{k}{2}t(1-t) |x-y|^2,
\]
for all $x,y\in \bR^d$ and $t\in [0,1]$.
The following proposition shows that if a function \( V:\mathbb{R}^d \to \mathbb{R} \) is \( k \)-convex in the classical sense, then the associated functional \( V(\mu) = \int_{\mathbb{R}^d} V(x)\, d\mu(x) \) is \( k \)-displacement convex on \( \mathcal{P}_{2,ac}(\mathbb{R}^d) \).

\begin{proposition}
    Let $V:\bR^d\to \bR$ be a $k$-convex function on $\bR^d$. Then, the functional 
    \[
    V(\mu)=\int_{\bR^d}V(x)\,d\mu(x)
    \]
    on $\cP_{2,ac}(\bR^d)$ satisfies the following: For any positive integer $N\ge2$,  probability measures $\mu_1,\ldots,\mu_N$ in $\cP_{2,ac}(\bR^d)$ and weights $\lambda_1,\ldots, \lambda_N > 0$ with $\sum_{i=1}^N \lambda_i = 1$, we have 
    \begin{equation}
        \label{eq_dis_convex}
    V(\bar{\mu})\le \sum_{i=1}^N\lambda_iV(\mu_i)-\frac{k}{2}\sum_{i=1}^N \lambda_iW_2^2(\bar{\mu},\mu_i),
    \end{equation}
    where $\bar{\mu}=\operatorname{bar}(\{(\mu_i,\lambda_i)\}_{i\in \llbracket N \rrbracket})$. In particular, if $k \ge 0$, then
    \begin{equation}
    \label{eq_convex}
         V(\bar{\mu})\le \sum_{i=1}^N\lambda_iV(\mu_i).
    \end{equation}
\end{proposition} 
\begin{proof}
    See \cite[Theorem 7.11]{MR3590527}.
\end{proof}
\begin{remark}
\label{rmk0114_1}

The functional
\[
V_2(\mu):= \int_{\bR^d}|x|^2\,d\mu(x)
\]
is a $2$-displacement convex functional. 
\end{remark}

\section{Proof of main results}\label{sec:proof}
Before proving the theorem, we present several lemmas that will be instrumental in its proof.
For $i,j\in \cV$ and nonnegative integers $m_1< m_2$, we say that $i$ meets $j$ on $\llbracket m_1,m_2\rrbracket$ if there exists a sequence $\{l_s\}_{s=m_1}^{m_2}$ with $l_{m_1}=i$ and $l_{m_2}=j$ such that 
\[
l_s\in \cN_{l_{s+1}}(s) \quad \text{for all $s=m_1,m_1+1,\ldots,m_2-1$}.
\]
We refer to such a sequence \( \{l_s\}_{s=m_1}^{m_2} \) as a \textit{joining sequence} from \( i \) to \( j \) on \( \llbracket m_1,m_2\rrbracket \). 

\begin{remark}
\label{rmk0114_2}
   Since we assume $i \in \cN_i(t)$ for all $t \in \bN_0$, it follows that every agent $i$ meets itself on $\llbracket m_1,m_2\rrbracket$ for all nonnegative integers $m_1<m_2$. In addition, if $i$ meets $j$ on $\llbracket m_1,m_2\rrbracket$ and $j$ meets $k$ on $\llbracket m_2,m_3\rrbracket$ ($m_2<m_3$), then $i$ meets $k$ on $\llbracket m_1,m_3\rrbracket$. Consequently, if $i$ meets $j$ on $\llbracket m_1,m_2\rrbracket$, then for any nonnegative integer $s_1\le m_1$ and $s_2\ge m_2$, $i$ meets $j$ on $\llbracket s_1,s_2\rrbracket$.
\end{remark}

The following lemma ensures that any pair of agents becomes connected over each time interval \( \llbracket t_k, t_{k+1}\rrbracket \), where $t_k:=\tau_{k(n-1)}$.
Here, \( \tau_k \) is the time index introduced in Section \ref{sec:pre}. This result will play a key role in establishing the forthcoming estimates.
\begin{lemma}
    \label{lem_meet}
    Let $t_k=\tau_{k(n-1)}$ for $k\in \bN_0$. For any $i,j\in \cV$ and $k\in \bN_0$, $i$ meets $j$ on $\llbracket t_k,t_{k+1}\rrbracket$. 
\end{lemma}

\begin{proof}
    Let $i,j\in \cV$ and set $\cV_0=\{i\}$, and define
    \[
    \cV_s =\{p\in \cV: \text{$i$ meets $p$ on $\llbracket0,s\rrbracket$}\} \quad \text{for $s\in \bN$}. 
    \]
    We claim that for all $l\in\bN_0$, if $\cV\setminus\cV_{\tau_l}\neq\varnothing$, then $\cV_{\tau_l}\subsetneq\cV_{\tau_{l+1}}$. 
    
    Since $\{\cG(\tau_l),\cG(\tau_l+1),\ldots,\cG(\tau_{l+1}-1)\}$ is jointly connected, there are some $p\in \cV_{\tau_l}$, $q\in \cV\setminus \cV_{\tau_l}$ and $m\in \llbracket\tau_l,\tau_{l+1}-1\rrbracket$ such that $(p,q)\in \cE(m)$, which implies $p\in \cN_q(m)$. 

    We construct a sequence $\{i_{\tau_l},i_{\tau_{l}+1},\ldots,i_{\tau_{l+1}}\}$ as follows: 
    \begin{equation*}
    i_{\tau_l+r}=\begin{cases}
    p\quad \text{for $r=0,1,\ldots,m-\tau_l$},\\
    q\quad \text{for $r=m-\tau_l+1,\ldots,\tau_{l+1}-\tau_l$}.
    \end{cases}
    \end{equation*}
    By employing this sequence, we see that $p$ meets $q$ on $\llbracket\tau_l,\tau_{l+1}\rrbracket$.
    Since $p\in \cV_{\tau_l}$, it follows that $i$ meets $p$ on $\llbracket0,\tau_{l}\rrbracket$.
    Hence, we obtain that $i$ meets $q$ on $\llbracket0,\tau_{l+1}\rrbracket$, which implies that $q\in \cV_{\tau_{l+1}}$.
    Since we choose $q$ in $\cV\setminus \cV_{\tau_l}$, it follows that $\cV_{\tau_l}\subsetneq \cV_{\tau_{l+1}}$.
    This proves the claim.

   By the claim, the finite sequence of sets \( \{\mathcal{V}_{\tau_l}\}_{l=0}^{n} \) is strictly increasing until it reaches \( \mathcal{V} \). Since the cardinality of \( \mathcal{V} \) is \( n \), we have
\[
\mathcal{V} = \mathcal{V}_{\tau_{n-1}}.
\]
    Therefore, we conclude that \( i \) meets \( j \) on \( \llbracket0, \tau_{n-1}\rrbracket = \llbracket0, t_1\rrbracket \).  
By repeating the same argument, one can show that \( i \) meets \( j \) on \( \llbracket t_k, t_{k+1}\rrbracket \) for all \( k \in \mathbb{N}_0 \).
\end{proof}

Define 
\[
V_2^{\max}(t)=\max_{i \in \cV} V_2(\mu_i(t)),
\]
where $V_2(\mu)=\int_{\bR^d}|x|^2\,d\mu$.

\begin{lemma}
\label{lem_V2converge}
There exists a constant  $V_2^{\max}(\infty)$ such that
 $\lim_{t\to \infty}V_2(\mu_i(t))=V_2^{\max}(\infty)$ for all $i\in \cV$.
\end{lemma}

\begin{proof}
We begin by noticing that
\[
0\le V_2^{\max}(t+1)\le V_2^{\max}(t).
\]
This inequality follows from the convexity \eqref{eq_convex} and Remark \ref{rmk0114_1}. Hence \( V_2^{\max}(t) \) is nonincreasing and bounded below, and thus, there exists a constant $V_2^{\max}(\infty)$ such that 
\begin{equation*}
V_2^{\max}(t)\searrow V_2^{\max}(\infty) \quad \text{as $t\to \infty$}.
\end{equation*}
Next, we show that $V_2(\mu_i(t))$ converges to $V_2^{\max}(\infty)$ for all $i\in \cV$.

Set $t_k=\tau_{k(n-1)}$ for $k\in \bN_0$.
For any $t\in \bN$, there exists $k\in \bN$ such that 
\[
t_{k-1}<t\le t_k<t_{k+1}.
\]
\
Fix $i\in \cV$, and choose arbitrary $j\in \cV$ (to be specified later).
By Lemma \ref{lem_meet}, $i$ meets $j$ on $\llbracket t_k,t_{k+1}\rrbracket$, and by Remark \ref{rmk0114_2}, we see that $i$ meets $j$ on $\llbracket t,t_{k+1}\rrbracket$.
Let $\{l_s\}_{s=t}^{t_{k+1}}$ be a joining sequence for $i$ and $j$ on $\llbracket t,t_{k+1}\rrbracket$.
Then for each $s\in \llbracket t,t_{k+1}-1\rrbracket$, using \eqref{eq_convex} and the lower bound assumption \eqref{eq_lower}, we have 
\begin{align*}   
V_2(\mu_{l_{s+1}}(s+1))
&\le
\sum_{r\in \cN_{l_{s+1}}(s)}w_{l_{s+1} r}(s)V_2(\mu_{r}(s))
\\&\le
(1-w_{l_{s+1}l_s}(s))V_2^{\max}(s)
+w_{l_{s+1}l_s}(s) V_2(\mu_{l_s}(s))
\\&
\le (1-\delta)V_2^{\max}(s)
+\delta V_2(\mu_{l_s}(s)).
\end{align*}
Here, we use the fact that $l_s\in \cN_{l_{s+1}}(s)$.

By iterating the above procedure and using the fact that \( V_2^{\max}(t) \) is decreasing, we obtain
\begin{align*}  
V_2(\mu_j(t_{k+1}))
&= V_2(\mu_{l_{t_{k+1}}}(t_{k+1}))
\\&\le 
(1 - \delta) V_2^{\max}(t_{k+1} - 1) + \delta V_2(\mu_{l_{t_{k+1} - 1}}(t_{k+1} - 1))
\\&
\le 
(1 - \delta^2) V_2^{\max}(t_{k+1} - 2) + \delta^2 V_2(\mu_{l_{t_{k+1} - 2}}(t_{k+1} - 2))
\\&
\le \cdots \le (1 - \delta^{t_{k+1} - t}) V_2^{\max}(t) + \delta^{t_{k+1} - t} V_2(\mu_{l_t}(t))
\\&
= (1 - \delta^{t_{k+1} - t}) V_2^{\max}(t) + \delta^{t_{k+1} - t} V_2(\mu_i(t))
\\&\le (1 - \delta^{t_{k+1} - t_{k-1}}) V_2^{\max}(t) + \delta^{t_{k+1} - t_{k-1}} V_2(\mu_i(t))
\\&
\le (1 - \delta^{M}) V_2^{\max}(t) + \delta^{M} V_2(\mu_i(t)),
\end{align*}
where \( M := \sup_{k \in \mathbb{N}} (t_{k+1} - t_{k-1}) \).  
Note that \( M \le 2L(n-1)\), where $L=\sup_i(\tau_i-\tau_{i-1})$ is the constant introduced in Section \ref{sec:pre}. 

We now choose $j \in\cV$ such that 
\[
V_2(\mu_j(t_{k+1}))=V_2^{\max}(t_{k+1}).
\]
Then the previous estimate yields
\[V_2^{\max}(t_{k+1})\le (1-\delta^{M})V_2^{\max}(t)+\delta^{M}V_2(\mu_{i}(t)).\]
Since $V_2^{\max}(\infty)\le V_2^{\max}(t_{k+1})$, it follows that
\begin{equation*}
V_2^{\max}(\infty)\le (1-\delta^{M})V_2^{\max}(t)+\delta^{M}V_2(\mu_{i}(t)).    
\end{equation*}
Moreover, as \( V_2(\mu_i(t)) \le V_2^{\max}(t) \), we conclude that
\[
V_2^{\max}(\infty) \le (1 - \delta^{M}) V_2^{\max}(t) + \delta^{M} V_2(\mu_i(t)) \le V_2^{\max}(t).
\]
Applying the squeeze theorem completes the proof.
\end{proof}

\begin{lemma}
    \label{lem_k-convex}
    Let $V:\bR^d \to \bR$ be a convex function. Assume that there exists a positive constant $C$ such that $-C\le V(x)\le C(1+|x|^2)$ for all $x\in \bR^d$. Set $V(\mu)=\int_{\bR^d} V(x) \,d\mu$. Then, there exists a constant $V^{\max}(\infty)$ such that
 $\lim_{t\to \infty}V(\mu_i(t))=V^{\max}(\infty)$ for all $i\in \cV$.
\end{lemma}
\begin{proof}
The proof proceeds analogously to that of Lemma \ref{lem_V2converge}, so, we omit it.  
\end{proof}

We now present the proof of Theorem \ref{thm_convergence}.

\begin{proof}[Proof of Theorem \ref{thm_convergence}]
We consider the first statement of the theorem. Namely, we show:
\begin{equation}
    \label{lem_convergence W_2}
    \text{For all $i,j\in \cV$, we have $\lim_{t\to \infty}W_2(\mu_i(t),\mu_j(t))=0$.}
\end{equation}

Fix $i\in \cV$. For $t\in \bN_0$ and $l\in \cN_i(t)$, by utilizing \eqref{eq_dis_convex} and Remark \ref{rmk0114_1}, we have
\begin{equation*}
    \begin{split}
    V_2(\mu_i(t+1))
    &\le
    \sum_{j\in \cN_i(t)}w_{ij}(t) V_2(\mu_j(t))-\sum_{j \in \cN_i(t)} w_{ij}(t)W_2^2(\mu_i(t+1),\mu_j(t)).
    \\
    &\le \sum_{j\in \cN_i(t)}w_{ij}(t) V_2(\mu_j(t))- w_{il}(t)W_2^2(\mu_i(t+1),\mu_l(t)).     
\end{split}
\end{equation*}
From this estimate and $w_{il}(t)\ge \delta$, it follows that for all $l\in \cN_i(t)$,
\begin{equation}
    \begin{split}
        \label{eq1209_1}
W_2^2(\mu_i(t+1),\mu_l(t))
&\le \delta^{-1}\sum_{j\in \cN_i(t)}w_{ij}(t) V_2(\mu_j(t))-\delta^{-1}V_2(\mu_i(t+1))
\\&= \delta^{-1}\sum_{j\in \cN_i(t)}w_{ij}(t) \left(V_2(\mu_j(t))-V_2(\mu_i(t+1))\right)
\\&\le \delta^{-1}\sum_{i,j\in \cV}|V_2(\mu_i(t))-V_2(\mu_j(t))|.
    \end{split}
\end{equation}
\[
\]

Let $\varepsilon>0$. By Lemma \ref{lem_V2converge}, $V_2(\mu_i(t))$ converges to $V_2^{\max}(\infty)$ for all $i\in \cV$. Hence there exists a positive integer $T$ such that  
\begin{equation}
\label{eq1209_2}
\sum_{i,j\in \cV}|V_2(\mu_i(t))-V_2(\mu_j(t))|<\varepsilon \quad \text{for $t=T,T+1,T+2,\ldots$.}
  \end{equation}
  
Let $i,j\in \cV$. Now, we show that $W_2(\mu_i(t), \mu_j(t)) \to 0$ as $t \to \infty$. 
Let $k_0$ be the smallest positive integer such that $T\le t_{k_0-1}$, where $t_u=\tau_{u(n-1)}$. 
Let $t\ge t_{k_0-1}$. Then for some $k \in \bN$,
\[
T\le t_{k_0-1}\le t_{k-1}<t\le t_{k}<t_{k+1}.
\]
By Lemma  \ref{lem_meet}, $i$ meets $j$ on $\llbracket t_k,t_{k+1}\rrbracket$, hence also on $\llbracket t,t_{k+1}\rrbracket$. 

Let $\{l_s\}_{s=t}^{t_{k+1}}$ be a joining sequence from $i$ to $j$ on $\llbracket t,t_{k+1}\rrbracket$.
Then, for all $s\in \llbracket t,t_{k+1}-1\rrbracket$, we have
\begin{equation*}
    W_2^2(\mu_{l_{s+1}}(s+1),\mu_{l_s}(s))\le \delta^{-1}\varepsilon
\end{equation*}
from the estimate \eqref{eq1209_1} with $l_s\in \cN_{l_{s+1}}(s)$ and \eqref{eq1209_2}.
By summing up, we arrive at 
\begin{align*}    
W_2(\mu_i(t),\mu_{j}(t_{k+1}))
&=W_2(\mu_{l_t}(t),\mu_{l_{t_{k+1}}}(t_{k+1}))
\\&\le \sum_{s=t}^{t_{k+1}-1}W_2(\mu_{l_{s+1}}(s+1),\mu_{l_s}(s))
\\&\le M\delta^{-1/2}\sqrt{\varepsilon},
\end{align*}
where $M = \sup_{k\in \bN} |t_{k+1}-t_{k-1}|$.
Similarly, we have 
\[
W_2(\mu_j(t),\mu_{j}(t_{k+1}))\le M\delta^{-1/2}\sqrt{\varepsilon}.
\]

Hence, for all $t\ge t_{k_0-1}$,
\[
W_2(\mu_i(t),\mu_{j}(t))\le  W_2(\mu_i(t),\mu_{j}(t_{k+1}))+W_2(\mu_j(t),\mu_{j}(t_{k+1}))\le 2M\delta^{-1/2}\sqrt{\varepsilon}.
\]
Since $\varepsilon > 0$ is arbitrary, \eqref{lem_convergence W_2} follows.

Next, we show the remaining part of Theorem \ref{thm_convergence}.
Let $i\in \cV$ and $p\in [1,2)$. By Lemma \ref{lem1211_2}, there exists a subsequence $\{\mu_i(t_k)\}$ of $\{\mu_i(t)\}$ such that 
\[
\lim_{k\to \infty} W_p(\mu_i(t_k),\mu_i^*)=0.
\]
for some $\mu_i^*\in \cP_2(\bR^d)$.
Assume that there exists another subsequence $\{\mu_i(s_k)\}$ such that $
\lim_{k\to \infty} W_p(\mu_i(s_k),\nu_i^*)=0$ for some $\nu_i^*\in \cP_2(\bR^d)$.

We claim that $\mu_i^*=\nu_i^*$. 
Let \( \varphi \in C_c^\infty(\mathbb{R}^d) \) with \( |D^2 \varphi| \le K \), for some $K>0$. Here, \( C_c^\infty(\mathbb{R}^d) \) denotes the space of infinitely differentiable functions with compact support.
Since $K|x|^2$ is $2K$-convex and $K|x|^2+\varphi(x)$ is $K$-convex on $\bR^d$, using Lemma \ref{lem_k-convex}, we see that
\[
\lim_{k\to \infty}\int_{\bR^d}K|x|^2\, d\mu_i(t_k)=\lim_{k\to \infty}\int_{\bR^d}K|x|^2\, d\mu_i(s_k),
\]
and
\[
\lim_{k\to \infty}\int_{\bR^d}\left(K|x|^2+ \varphi(x)\right)\, d\mu_i(t_k)=\lim_{k\to \infty}\int_{\bR^d}\left(K|x|^2+ \varphi(x)\right)\, d\mu_i(s_k).
\]
From this, we obtain that
\begin{align*}    
\int_{\bR^d} \varphi(x) \,d\mu_i^*
&=\lim_{k\to \infty}\int_{\bR^d} \varphi(x)\,d\mu_i(t_k)
\\&=\lim_{k\to \infty}\int_{\bR^d}\left(K|x|^2+ \varphi(x)\right)\, d\mu_i(t_k)-\lim_{k\to\infty}\int_{\bR^d} K|x|^2 \,d\mu_i(t_k)
\\&
=\lim_{k\to \infty}\int_{\bR^d}\left(K|x|^2+ \varphi(x)\right) \,d\mu_i(s_k)-\lim_{k\to\infty}\int_{\bR^d} K|x|^2 \,d\mu_i(s_k)
\\&
=\lim_{k\to \infty}\int_{\bR^d} \varphi(x)\,d\mu_i(s_k)
\\&
=\int_{\bR^d} \varphi(x)\,d\nu_i^*.
\end{align*}
Since \( \varphi \) is an arbitrary test function in \( C_c^\infty(\mathbb{R}^d) \), it follows that \( \mu_i^* = \nu_i^* \).
This implies that
\begin{equation}
    \label{eq1211_4}
    \lim_{t\to \infty} W_p(\mu_i(t),\mu_i^*) = 0.
\end{equation}
Suppose, for contradiction, that \eqref{eq1211_4} does not hold. 
Then there exists a sequence $\{s_m\}$ such that
\begin{equation}
    \label{eq_contrad}
    W_p(\mu_i(s_m),\mu_i^*) \ge \gamma,
\end{equation}
for some constant $\gamma>0$.
Since
\[
\sup_{m\in \mathbb{N}_0} \int_{\mathbb{R}^d} |x|^p \, d\mu_i(s_m) < \infty,
\]
we can extract a subsequence $\{s_{m_k}\}$ (denoted simply by $s_k$) that converges in $W_p$. By the preceding claim, the limit of $\mu_i(s_k)$ must be $\mu_i^*$. This contradicts \eqref{eq_contrad}.

Finally, we show that $\mu_i^*=\mu_j^*$ for all $i,j\in \cV$.
Recall that for any $1\le p_1\le p_2<\infty$, 
\[
W_{p_1}(\mu,\nu)\le W_{p_2}(\mu,\nu).
\]
Since $1 \le p < 2$, this gives
\begin{align*}   
W_p(\mu_i^*,\mu_j^*)
&\le W_p(\mu_i^*,\mu_i(t))+W_p(\mu_i(t),\mu_j(t))
+W_p(\mu_j(t),\mu_j^*) 
\\&
\le W_p(\mu_i^*,\mu_i(t))+W_2(\mu_i(t),\mu_j(t))
+W_p(\mu_j(t),\mu_j^*).
\end{align*}
By letting $t\to \infty$ to above inequalities, from  \eqref{lem_convergence W_2} and \eqref{eq1211_4}, we arrive at 
\[
W_p(\mu_i^*,\mu_j^*)
\le \lim_{t\to \infty}\left( W_p(\mu_i^*,\mu_i(t))+W_2(\mu_i(t),\mu_j(t))
+W_p(\mu_j(t),\mu_j^*)\right)=0.
\]
From this and \eqref{eq1211_4}, we obtain \eqref{main converge}. This completes the proof.
\end{proof}

The proof of Theorem \ref{thm_convergence2} is analogous.
\begin{proof}[Proof of Theorem \ref{thm_convergence2}]
    Since the underlying idea of the proof is essentially the same as that of the proof of Theorem \ref{thm_convergence}, we present only a concise argument. 
    
    Since 
    \[
    \sup_{i\in\cV}\int_{\bR^d}|x|^q\,d\mu_i(0)<\infty
    \]
    and the functional $V(\mu):=\int|x|^q\,d\mu$ satisfies \eqref{eq_convex}, it follows that
    \[
    \sup_{t\in \bN_0}\int_{\bR^d}|x|^q\,d\mu_i(t)<\infty, \quad t=0,1,\ldots.    
    \]
    By Lemma \ref{lem1211_2}, for each $i$, we have a sequence $\{t_k\}$ and $\mu_i^*\in \cP_2(\bR^d)$ such that 
    \[
    \lim_{k\to \infty}W_2(\mu_i(t_k),\mu_i^*)=0.
    \]

    By following the proof of Theorem \ref{thm_convergence} line by line, we obtain
    \[
    \lim_{t\to \infty}W_2(\mu_i(t),\mu_i^*)=0,
    \]
    for all $i\in \cV$.
    This and  \eqref{lem_convergence W_2} imply 
    \[
    W_2(\mu_i^*,\mu_j^*)=\lim_{t\to\infty }W_2(\mu_i(t),\mu_j(t))=0.
    \]
    Hence, we have \eqref{main converge2}.
    This completes the proof.
\end{proof}

\section{Conclusion}\label{sec:conclusion}
In this paper, we study the convergence of consensus dynamics governed by an update rule based on the Wasserstein barycenter. 
In dimensions greater than one, the lack of convexity of the Wasserstein distance prevents the use of arguments employed in previous works \cite{bishop2014distributed,bishop2021network,MR3373939}. 
To address this issue, we introduce a suitable convex functional and apply a Wasserstein version of Jensen's inequality. 
Under assumptions on network connectivity and interaction weights, we establish the convergence to consensus.

\section*{Acknowledgment}
This work was supported by the National Research Foundation of Korea (NRF) (No. RS-2022-NR069609).

\end{document}